\documentclass{amsart}

\usepackage{xcolor}

\usepackage{amsfonts}
\usepackage{amssymb}
\usepackage{graphicx}
\usepackage{mathrsfs}
\usepackage{amscd}
\usepackage{hyperref}
\usepackage{stmaryrd}
\usepackage{caption}
\usepackage{subcaption}
\usepackage{color}
\usepackage{xcolor}
\usepackage{comment}
\usepackage{bm,bbm}
\usepackage{stmaryrd}
\usepackage{euscript}
\usepackage{microtype}
\usepackage{enumitem}

\usepackage{tikz} 
\usepackage{tikz-cd}
\usepackage{mathtools}

\newtheorem{theorem}{Theorem}[section]

\newtheorem{corollary}[theorem]{Corollary}

\newtheorem{prop}[theorem]{Proposition}
\newtheorem{defn}[theorem]{Definition}
\newtheorem{example}[theorem]{Example}
\newtheorem{remark}[theorem]{Remark}

\numberwithin{equation}{section}

\newcommand{\bS}{\mathbb{S}}

\newcommand{\End}{\mathrm{End}}

\newcommand{\Id}{\mathrm{Id}}

\newcommand{\Ad}{\mathrm{Ad}}

\title{Moment Lagrangians, unobstructedness and symplectic groupoids}

\author[Li]{Yan-Lung Leon Li}
\address{Center of Geometry and Physics, Institute for Basic Science (IBS), Pohang 37673, Korea}
\email{ylli@ibs.re.kr}

\begin{document}
\maketitle

\begin{abstract}
Moment Lagrangian $L_\mu$ is a Lagrangian in $T^*G^- \times Y^- \times Y$ associated to a Hamiltonian $G$-space $Y$ with a moment map $\mu$. In this paper, we prove that $L_\mu$ is tautologically unobstructed under mild assumptions on $Y$. 

As a key ingredient in the proof, we constructed a new symplectic groupoid structure on $T^*G^- \times Y^- \times Y$ over $G\times Y$ for which $L_\mu$ is simultaneously the unit and the fixed locus of the inversion,  which might be of independent interest. 
\end{abstract}

\section{Introduction}
Symplectic manifolds are the fundamental geometric objects underlying Hamiltonian mechanics, whose origin can be traced back to the classical works of Lagrange. Through the important works of Kostant, Kirillov and Souriau in the 1960s, it became clear that the natural geometric framework for Hamiltonian mechanics with symmetry group $G$ (assumed to be compact and connected throughout the article) is provided by Hamiltonian $G$-spaces $(Y, \omega, \mu)$, where $(Y,\omega)$ is a symplectic manifold  with a Hamiltonian $G$-action in the sense that it admits a $G$-equivariant moment map $\mu$ satisfying the Hamilton's equation \ref{hamilton}.

Meanwhile, in \cite{Weinsteinsympgeom}, Weinstein proposed his famous symplectic creed, namely ``Everything is a Lagrangian submanifold", which led to his formulation of the symplectic ``category" $\mathrm{Symp}$. 
Roughly speaking, its objects are symplectic manifolds, its morphisms are Lagrangian correspondences between them, and the composition of morphisms is given by geometric composition of correspondences. In this framework, Hamiltonian $G$-spaces $(Y, \omega, \mu)$ can be incorporated naturally through their moment Lagrangians
$$L_\mu \coloneqq \{(g, \mu(y), y, g\cdot y) | g\in G; y \in Y\} \subset (T^*G)^- \times Y^- \times Y$$

as a Lagrangian correspondence from $(T^*G,-\omega_{can}) $ (see Remark \ref{sign} for more explanations on the sign) to  $(Y^- \times Y, (-\omega) \oplus \omega)$, providing a realization of the action as a morphism in $\mathrm{Symp}$. This perspective has led to profound applications in quantization and semi-classical analysis; see, for instance, \cite{GSsemiclassical} and the references therein.

On the other hand, in his ICM address \cite{Teleman}, Teleman proposed using $L_\mu$ to construct a 
topological $G$-action on the Fukaya category $\mathrm{Fuk}(Y)$ of $Y$. A key intermediate step is to use $L_\mu$ 
to induce a functor from the wrapped Fukaya category of $T^*G$ to the Fukaya category of $Y^-\times Y$, 
which the latter acts on $\mathrm{Fuk}(Y)$ via geometric transformations. Teleman's proposal has since inspired substantial interests in the Lagrangian Floer theory of $L_\mu$ and its applications, resulting in important developments in the monotone setting, including \cite{EL, BCL, Chowmomentcorr} to name a few.

Meanwhile, to the best of the author's knowledge, the Floer theory of $L_\mu$ for more general symplectic manifolds, in terms of foundational results of Fukaya, Oh, Ohta and Ono in their series of works on Lagrangian Floer theory (e.g. \cite{FOOO,foookuranishi}), seems to remain largely unexplored. In particular, a fundamental question is to determine whether or not $L_\mu$ is (Floer-theoretically) unobstructed, in the sense that its (Lagrangian) Floer complex $CF(L_\mu)$ is an unobstructed $A_\infty$ algebra (in particular a chain complex) after some bulk or boundary deformations (recalled in Remark \ref{bdydeform}) so that its Floer cohomology $HF(L_\mu)$ is well-defined after such deformations. Establishing unobstructedness is therefore a necessary first step towards any applications of the Floer theory of $L_\mu$.

In this work, we prove the tautogical unobstructedness of $L_\mu$ under some mild assumptions as follows:

\begin{theorem}\label{momentlagrangianunobstructed}
Given a closed, connected, spin Hamiltonian $G$-space $(Y, \omega, \mu)$, its moment Lagrangian $L_\mu$ is tautologically unobstructed in $X \coloneqq (T^*G)^- \times Y^- \times Y$. 
\end{theorem}

\begin{remark}
In \cite[Lemma 6.1.1]{EL}, Evans and Lekili established the monotonicity of $L_\mu$ under the additional assumptions that $Y$ is simply-connected and monotone, in order to construct an $A_\infty$ functor from the wrapped Fukaya category of $T^*G$ to the monotone Fukaya category $\mathrm{Fuk}(Y^- \times Y)$. This has played a pivotal role in computing $\mathrm{Fuk}(Y)$ in terms of determining its split generators. 

From this perspective, Theorem \ref{momentlagrangianunobstructed} shows that the unobstructedness of $L_\mu$ persists beyond monotone setting. Moreover, combined with Fukaya's result in \cite[Theorem 1.7]{Fukaya-corr}, it yields an $A_\infty$ functor between their (compact, unobstructed) Fukaya categories. It would be interesting to see if this functor would also be useful in computing $\mathrm{Fuk}(Y)$ in the sense of \cite{AFOOO}.
\end{remark}

\begin{remark}
    This result should not be confused with the unobstructedness result of Xiao in \cite{xiao2}. In the loc. cit, the term ``moment Lagrangian correspondence" refers to the correspondence between a Hamiltonian $G$-space $Y$ and its symplectic quotient $Y \sslash_0 G$ (which is called ``moment level correspondence" in \cite{lauleungli}).
\end{remark}

As an application of Theorem \ref{momentlagrangianunobstructed}, together with Fukaya's unobstructedness theorem on geometric transformations, we obtain the unobstructedness of the following Lagrangians:

\begin{corollary}\label{cor: exunoblag}
    \begin{enumerate}
        \item Given a closed, connected, spin Hamiltonian $G$-space $(Y, \omega, \mu)$ such that $G$ acts freely on $\mu^{-1}(0)$, the following Lagrangian is unobstructed: $$L' \coloneqq  \{(y, g\cdot y) \in Y^- \times Y | g\in G; y\in \mu^{-1}(0)\} \subset Y^- \times Y.$$
        \item Given a spin coadjoint orbit $(\mathcal{O}_\lambda, \omega) \subset \mathfrak{g}^*$ endowed with the Kirillov-Kostant-Souriau symplectic form $\omega$, the following Lagrangian is unobstructed: $$L_{\mathcal{O}_\lambda} \coloneqq  \{(g, \xi) \in G \times \mathfrak{g}^* | g\in G; \Ad^*_g(\xi) = \xi\in \mathcal{O}_\lambda\}\subset G \times \mathfrak{g}^*.$$
    \end{enumerate}
\end{corollary}

Two remarks of Corollary \ref{cor: exunoblag} are as follows:

\begin{remark}
In the monotone setting, $L'$ was also studied by Evans and Lekili in \cite[Section 6.2]{EL}, where $L'$ is assumed to be monotone. 
From this perspective, (1) can be viewed as showing that $L'$ is in fact unobstructed in general. 
\end{remark}

\begin{remark}
In (2), $L_{\mathcal{O}_\lambda}$ is the character Lagrangian introduced by Gullemin and Sternberg in \cite{GSmomentred} with applications to representation theory and quantization. For more recent treatments, see, for instance, \cite{GSsemiclassical, CWsympcat}. 

Meanwhile, to the best of author's knowledge, its Floer theory remains largely unexplored except for the following special case: when $G = SU(2)$ and $\mathcal{O}_\lambda = \bS^2$, the character Lagrangian $L_{\mathcal{O}_\lambda} = \bS^1 \times \bS^2$ can be identified with a monotone Lagrangian of Oakley-Usher type (a.k.a. generalised Polterovich Lagrangian) in $T^*\bS^3$, studied in \cite{AlbersFrauenfelder, OU, AbouzaidDiogo, BP} to name a few. 

In view of this, (2) establishes the unobstructedness of $L_{\mathcal{O}_\lambda}$ whenever $\mathcal{O}_\lambda$ is spin, not necessarily being a monotone Lagrangian (since $(\mathcal{O}_\lambda, \omega)$ needs not be monotone).
\end{remark}

Returning to the proof of Theorem \ref{momentlagrangianunobstructed}, the key idea is to realize $L_\mu$ as the fixed locus of an anti-symplectic involution $i$. Its unobstructedness then follows from a theorem of Fukaya, Oh, Ohta and Ono \cite[Corollary 1.7]{FOOOinvol}. 
Interestingly, $L_\mu = \mathrm{Fix}(i)$ admits a broader interpretation in terms of symplectic groupoids, which we briefly recalled below.

Symplectic groupoid was introduced by Weinstein in \cite{weinsteinsympgroupoid}, with a more detailed treatment appearing in \cite{sympgroupoid}. Roughly speaking, a symplectic groupoid is a Lie groupoid $X$ over $M$ such that $X$ admits a multiplicative symplectic structure integrating the Poisson structure on $M$.  The notion arises naturally in the study of symplectic realizations and quantization of Poisson manifolds. 

In this work, we construct a new example of symplectic groupoid on $X \coloneqq (T^*G)^- \times Y^- \times Y$ for which the moment Lagrangian $L_\mu$ is the unit, a result which might be of independent interest, as follows:

\begin{theorem}(= Theorem \ref{twisted-groupoid}) \label{intro:twisted-groupoid}
    Given a Hamiltonian $G$-space $(Y, \omega, \mu)$, there exists a symplectic groupoid structure on $X$ over $G \times Y$
 such that $L_\mu$ is the unit and also the fixed locus of the anti-symplectic involution $i$:
    $$L_\mu = \underline{e} = Fix(i).$$
\end{theorem}

\begin{remark}
    The equality $L_\mu = \underline{e}$ already appears in \cite{Weinsteinsympgeom}. The theorem above provides a new interpretation of this from the perspective of symplectic groupoids.
\end{remark}

\begin{remark}
As another consequence of \cite[Corollary 1.6]{FOOOinvol} applied to $L_\mu = Fix(\underline{e})$, $HF(L_\mu)$ admits a graded-commutative product induced by $m_2$ (after suitably twisted by some signs). 

Another interesting way to see this is via $L_\mu  = \underline{e}$ in Theorem \ref{intro:twisted-groupoid}: graded commutativity of $HF(L_\mu)$ follows from a 
general Eckmann-Hilton-type argument applied to $\End(\underline{e})$, where $\underline{e}$ is a unit object in the monoidal Fukaya category of a symplectic groupoid, due to Pascaleff in \cite[Section 5.1]{Pascaleffmonoidal}. This provides another instance in which the structural properties of symplectic groupoids implies the graded commutativity of the Floer cohomology of units.
\end{remark}

Returning to the proof of Theorem \ref{intro:twisted-groupoid}, the groupoid structure is given explicitly in 
Theorem \ref{twisted-groupoid}, which can be checked to satisfy the axioms of Lie groupoid (recalled in Appendix \ref{appendix:liegroupoid}). The multiplicativity of the symplectic form follows from the properties of moment map, 
although the derivation is somewhat tedious, as demonstrated in the proof of Theorem \ref{twisted-groupoid}. Perhaps the most mysterious part of the construction is the groupoid structure itself, even to the author.

\subsection*{Organization of this paper} In Section 2, we recall the notions of moment Lagrangian and symplectic groupoid. Then we prove Theorem \ref{intro:twisted-groupoid} (= Theorem \ref{twisted-groupoid}) on the construction of the symplectic groupoid on $X$.

In Section 3, we first briefly recall the notions of Lagrangian Floer complexes and their unobstructedness. Then we prove Theorem \ref{momentlagrangianunobstructed} on the unobstructedness of $L_\mu$. 
Afterwards, we prove Corollary \ref{cor: exunoblag} on unobstructedness of some geometric transformations by $L_\mu$.

In Appendix \ref{appendix:symplecticform}, we recall an explicit formula for the symplectic form $\Omega$ on $G \times \mathfrak{g}^*$.

In Appendix \ref{appendix:liegroupoid}, we recall the notions of Lie groupoid to clarify the choices and conventions for this paper. Other conventions are listed in the following subsection.

\subsection*{Notations and Conventions}
Throughout the whole paper, we will generally follow the conventions in \cite{CFM}. Some of the relevant ones are as follows:
\begin{enumerate}
\item For any smooth manifold $M$, its cotangent bundle $T^*M$ has a canonical symplectic form $\omega_{can} = -d\lambda_{can}$, i.e. locally of the form $\sum_{i} dq_i \wedge dp_i$ with respect to local coordinates $(q_i)$ of $M$ and cotangent fiber coordinates $(p_i)$.


\item Given a Hamiltonian $G$-space $(Y, \omega, \mu)$, the moment map $\mu$ is a $G$-equivariant map $\mu: Y \rightarrow \mathfrak{g}^*$, where $\mathfrak{g}^*$ is the dual of the Lie algebra of $G$ endowed with coadjoint $G$-action, satisfying the following Hamilton's equation:
\begin{equation}\label{hamilton}
  \iota_{v^\#}\omega = \langle d\mu(-), v \rangle, \forall v\in \mathfrak{g}.
\end{equation}

where $v^\#$ is the infinitesimal generator of the $G$-action on $Y$ associated to $v\in \mathfrak{g}$ defined as $v^\#(y) = \frac{d}{dt}|_{t=0} \exp(-tv)\cdot y$ for any $y\in Y$, which gives rise to a Lie algebra homomorphism from $\mathfrak{g}$ to the Lie algebra of vector fields on $Y$.

Under this convention, the $G$-equivariance of $\mu$ is equivalent to that the comoment map $\mu^*: \mathfrak{g} \rightarrow C^\infty(Y)$, given by $\mu^*(v) \coloneqq \langle \mu(-), v\rangle$, is a Lie algebra homomorphism.



\end{enumerate}

\subsection*{Acknowledgments}
The author thanks Hanwool Bae, 	
Sam Bardwell-Evans, Siu-Cheong Lau, Conan Leung, Yong-Geun Oh and Yutung Yau for helpful discussions at various stages of this project. He is also grateful to Ziming Ma for his hospitality during the International Conference on Mathematics Inspired by String Theory, where part of this work was carried out and presented.
This work was supported by the Institute for Basic Science (IBS-R003-D1).

\subsection*{AI Declaration} ChatGPT-6 Astra was used solely to improve the grammar and clarity of the exposition. All the (original) mathematical statements and proofs are the author's own work, part of which had already been presented before as stated in the acknowledgment. The author takes full responsibility for their validity.

\section{Moment Lagrangians and symplectic groupoids}

In this section, we first recall the notion of moment Lagrangian and symplectic groupoid, and then construct a new example of symplectic groupoid whose unit is the moment Lagrangian.

The precise definition of moment Lagrangian is as follows:

\begin{defn}\label{def:momentlag}
    Given a Hamiltonian $G$-space $(Y, \omega)$ with moment map $\mu$, its moment Lagrangian $L_\mu$ is defined as 
    $$L_\mu = \{(g, \mu(y), y, g\cdot y) | g\in G; y \in Y\} \subset ((G \times \mathfrak{g}^*)^- \times Y^- \times Y, (-\Omega) \oplus (-\omega) \oplus \omega).$$
    where $Y^-$ is the symplectic manifold $(Y, -\omega)$, and $\Omega$ is the symplectic form on $G \times \mathfrak{g}^*$ corresponding to $-\omega_{can}$ under trivialization $G \times \mathfrak{g}^* \cong T^*G$ by left translation (with its explicit formula recalled in Appendix \ref{appendix:symplecticform}).
\end{defn}

\begin{remark}\label{sign}
This choice of symplectic structure $\Omega$ on $G \times \mathfrak{g}^*$ is for the compatibility with the chosen conventions for symplectic groupoids. See Appendix \ref{appendix:symplecticform} for further details.
\end{remark}

    

We now recall the notion of symplectic groupoids as follows:

\begin{defn}\label{def:sympgroupoid}
A symplectic groupoid is a Lie groupoid
$(X \xRightarrow[t]{s} M, m, e, i)$, where $(X, \omega)$ is a symplectic manifold such that $\omega$ is multiplicative: 
\[
m^*\omega = pr_1^*\omega+ pr_2^*\omega \in \Omega^2(X^{(2)}),
\]
where $pr_j: X^{(2)} \rightarrow X$ be the projection to $j$-th component.
\end{defn}


\begin{prop} \label{lagrangiansubmanifold}
Given a symplectic groupoid $(X, \omega) \xRightarrow[t]{s} M$, the following are Lagrangians:
\begin{enumerate}
\item $\underline{m} := Graph(m) = \{(p,q,p\cdot q) \ |\,  (p,q)\in X^{(2)}\} \subset X^- \times X^- \times X,$
\item $\underline{e} := e(M)=\{e(x) |\,  x\in M\} \subset X,$
\item $\underline{i} := Graph(i) =\{(p,p^{-1})) \in X \times X|\,  p\in X \}\subset X \times X.$ 
\end{enumerate}
\end{prop} 

\begin{proof}
See e.g. \cite[Prop. 14.9]{CFM}.
\end{proof}

We now state the main result of this paper, which constructs a new example of symplectic groupoid whose unit is the moment Lagrangian:

\begin{theorem}\label{twisted-groupoid}
    Given a Hamiltonian $G$-space $(Y, \omega, \mu)$, there exists a symplectic groupoid structure on $(X \coloneqq (G \times \mathfrak{g}^*)^- \times Y^- \times Y, \omega_X \coloneqq ((-\Omega)\oplus (-\omega) \oplus \omega))$:
    $$((X , \omega_X) \xRightarrow[t]{s} G \times Y, m, e, i),$$
given by the following:
    \begin{itemize}
        \item  $s(g, \xi, y, z) = (g, g^{-1}\cdot z)$; $t(g, \xi, y, z) = (g, y)$;
        \item $(g_1, \xi_1, y_1, z_1) \cdot (g_2, \xi_2, y_2, z_2) = (g, \xi_1 + \xi_2 - \mu(y_2), y_1, z_2)$, where 
        $$g_1 = g_2 = g; z_1 = g\cdot y_2;$$
        \item $e(g, x) = (g, \mu(x), x, g\cdot x)$;
        \item $i(g, \xi, y, z) = (g, \mu(g^{-1}\cdot z)+ \mu(y) - \xi, g^{-1}\cdot z, g\cdot y)$.
    \end{itemize}

     In particular, $L_\mu$ is the unit and also the fixed locus of the anti-symplectic involution $i$:
    $$L_\mu = \underline{e} = Fix(i).$$
\end{theorem}

\begin{remark}
Combining Theorem \ref{twisted-groupoid} with Proposition \ref{lagrangiansubmanifold}, we obtain an alternative proof of the fact that the moment Lagrangian $L_\mu$ is a Lagrangian of $(X, \omega_X)$.
\end{remark}

\begin{proof}
It is straightforward to check that the above maps satisfy the axioms of a Lie groupoid, and that $L_\mu = \underline{e} = Fix(i)$. It remains to check that $\omega_X$ is multiplicative, which will be done by definition as follows.

    First of all, the tangent space of $X^{(2)}$ at $(x_1, x_2) \coloneqq ((g, \xi_1, y_1, z_1), (g, \xi_2, y_2, z_2))$ is given by 
    $$T_{(x_1, x_2)}X^{(2)} = \{((Z, \eta_1, v_1, v'_1), (Z, \eta_2, v_2, v'_2))| Z\in \mathfrak{g}; v'_1 = g_* v_2 + (y_2)_* Z\}.$$
    where $g_*: T_{y}Y \rightarrow T_{g\cdot y}Y$ is the differential of the action by $g$ and $y_*: \mathfrak{g} \rightarrow T_{y}Y$ is the infinitesimal generator $Z^\#$ evaluated at $y$.

    Also, the differential of $m$ at $(x_1, x_2)$ is given by
 $$dm_{(x_1, x_2)}((Z, \eta_1, v_1, v'_1), (Z, \eta_2, v_2, v'_2))=(Z, \eta_1 + \eta_2 - d\mu(v_2), v_1, v'_2).$$

    Therefore, the symplectic form is multiplicative if and only if for any 
    $$((Z, \eta_1, v_1, v'_1), (Z, \eta_2, v_2, v'_2)), (W, \xi_1, w_1, w'_1), (W, \xi_2, w_2, w'_2))\in T_{(x_1, x_2)}X^{(2)},$$
    the following holds:
\begin{align*}
    &\omega_X((Z, \eta_1 + \eta_2 - d\mu(v_2), v_1, v'_2), (W, \xi_1 + \xi_2 - d\mu(w_2), w_1, w'_2))\\
    = &\omega_X((Z, \eta_1, v_1, v'_1), (W, \xi_1, w_1, w'_1)) + \omega_X((Z, \eta_2, v_2, v'_2), (W, \xi_2, w_2, w'_2)).
\end{align*}
Using \ref{symplecticform}, the above is equivalent to the following:
\begin{align*}
&-\langle d\mu(v_2), \mathcal{L}^{-1}_{g}(W) \rangle + \langle d\mu(w_2), \mathcal{L}^{-1}_{g}(Z) \rangle + \langle \mu(y_2), [\mathcal{L}^{-1}_{g}(Z), \mathcal{L}^{-1}_{g}(W)] \rangle \\
= &-\omega(g_*(v_2), (y_2)_*W) - \omega((y_2)_*Z, g_*(w_2)) - \omega((y_2)_*Z, (y_2)_*W).
\end{align*}
The above equality holds if the following identities hold:
\begin{align*}
\langle d\mu(v), \mathcal{L}^{-1}_{g}(W) \rangle &= \omega(g_*(v), (y_2)_*W);\\
\langle \mu(y_2), [\mathcal{L}^{-1}_{g}(Z), \mathcal{L}^{-1}_{g}(W)] \rangle &= - \omega((y_2)_*Z, (y_2)_*W).
\end{align*}
The first identity follows from the Hamilton's equation \ref{hamilton}, and the second identity follows from that the comoment map $\mu^*$ is a Lie algebra homomorphism.
\end{proof}

\section{Applications to unobstructedness}

In this section, we provide some applications of Theorem \ref{twisted-groupoid} to obtain results on (Lagrangian) Floer theory of $L_\mu$ and some geometric transformations by it. 

\subsection{Floer theory and unobstructedness}
In this section, we briefly review the notions of Floer complex and its unobstructedness. Readers are referred to e.g. \cite{Fukaya-corr, FOOO,FOOOinvol} for more details.

Given a closed (or tame) symplectic manifold $(X, \omega)$ and a closed, connected, oriented and relatively spin Lagrangian $L$ in $X$, the Floer complex $CF(L)$ of $L$ is a filtered $A_\infty$ algebra $(CF(L), \{m_k\}_{k\geq 0})$ over the (universal) Novikov ring 
$$\Lambda_0 \coloneqq \{\sum_{i=0}^\infty a_i T^{\lambda_i}| a_i \in \mathbb{R}; 0=\lambda_0<\lambda_1<\cdots; \lambda_i \rightarrow \infty\},$$
 whose structure maps $m_k: CF(L)^{\otimes k} \rightarrow CF(L)$ are defined by counting pseudo-holomorphic discs with boundary on $L$ and satisfy the $A_\infty$ relations. In particular, if the curvature term $m_0$ vanishes, then $m_1 \circ m_1 = 0$, and hence $(CF(L), m_1)$ is a chain complex with cohomology $HF(L)$ being the Floer cohomology of $L$.

In general, $m_0$ may not vanish, which is an obstruction to $(CF(L), m_1)$ being a chain complex. However, $CF(L)$ could become unobstructed after deforming the $A_\infty$ algebra by a bounding cochain $b$ defined as follows:

\begin{defn}\label{unobstructedlag}
$L$ is called unobstructed if $(CF(L), \{m_k\})$ is unobstructed as a filtered $A_\infty$ algebra, i.e. there exists a bounding cochain $b\in CF^1(L; \Lambda_+)$ (where $\Lambda_+ \subset \Lambda_0$ is the (unique) maximal ideal) such that the Maurer-Cartan equation holds:
$$\sum_{k=0}^{\infty} m_k(b, \ldots, b) = 0.$$

In particular, $L$ is called tautologically unobstructed if $0$ is a bounding cochain, i.e. $m_0(1) = 0$.
\end{defn}

\begin{remark}\label{bdydeform}
    Any bounding cochain $b$ defines a deformed $A_\infty$ algebra $(CF(L), \{m_k^b\})$ with $m_0^b(1) \coloneqq \sum_{k=0}^{\infty} m_k(b, \ldots, b)  = 0$ by assumption, and hence $(CF(L), m_1^b)$ is a chain complex with cohomology $HF(L, b)$ being the deformed Floer cohomology of $(L, b)$. See \cite[Definition 6.3]{FOOOinvol} and the reference therein for more details.
\end{remark}

\subsection{Unobstructedness of moment Lagrangians}

In this section, we prove Theorem \ref{momentlagrangianunobstructed} on tautological unobstructedness of the moment Lagrangian using Theorem \ref{twisted-groupoid} as follows:

\begin{proof}[Proof of Theorem \ref{momentlagrangianunobstructed}]
    The assumptions guarantee that $X$ is tame and spin. Also, the moment 
    Lagrangian $L_\mu \cong G \times Y$ is closed, connected, oriented and spin (and hence is $i$-relatively spin: See \cite[Section 3.2]{FOOOinvol}). By Theorem \ref{twisted-groupoid}, $L_\mu = Fix(i)$ is the fixed locus of an anti-symplectic involution $i$; also, the induced map $\iota_*: \pi_1(L_\mu)\rightarrow \pi_1(X)$ of the inclusion $\iota: L_\mu \hookrightarrow X$ is injective, since $\iota$ admits a left inverse given by 
    $e \circ t: X \rightarrow L_\mu$. Therefore, by \cite[Corollary 1.6]{FOOOinvol} (in which $b$ can be chosen to be zero as shown in its proof), $L_\mu$ is tautologically unobstructed.
\end{proof}

\subsection{Unobstructedness of geometric transformations}

In this section, we provide more examples of unobstructed Lagrangians as geometric transformations of $L_\mu$ as an unobstructed Lagrangian correspondence.

Firstly, we recall the notions of Lagrangian correspondences and geometric transformations by them, which are first introduced by Weinstein in his symplectic category formalism in \cite{sympcategory}:

\begin{defn}
Given a Lagrangian correspondence $\mathcal{L}$ from $(M, \omega_M)$ to $(N, \omega_N)$, i.e. a Lagrangian in $(M^- \times N, -\omega_M\oplus\omega_N)$. For any Lagrangian $L \subset (M, \omega_M)$, the geometric transformation of $L$ by $\mathcal{L}$, denoted as $\mathcal{L} \circ L$, is a subset of $N$ defined as
\begin{equation}
\mathcal{L} \circ L = pr_{N}((L \times \mathcal{L})\cap (\Delta_{M} \times N)),
\end{equation}
where $pr_{N} :M \times M \times N \rightarrow N$ is the natural projection.\\

$\mathcal{L} \circ L$ is called embedded if 
   $(L \times \mathcal{L})\cap (\Delta_{M} \times N)$ is a transverse intersection in $M \times M \times N$, and    
$pr_{N}$ restricted to a smooth embedding on this intersection. If so, then $\mathcal{L} \circ L$ is an embedded Lagrangian of $(N, \omega_N)$.
\end{defn}

Examples of geometric transformations by $L_\mu$ (and its transpose $L^t_\mu$) are given as follows:

\begin{example}\label{example:shift-diagonal}
The moment Lagrangian $L_\mu \subset (G \times \mathfrak{g}^*)^- \times Y^- \times Y$ is a Lagrangian correspondence from $(G \times \mathfrak{g}^*, \Omega)$ to $(Y^- \times Y, -\omega \oplus \omega)$. Let $L \coloneqq G \times {0} \subset (G \times \mathfrak{g}^*)$ be the zero section, then the geometric transformation of $L$ by $L_\mu$ is given by 
$$L' \coloneqq L_\mu \circ L =  \{(y, g\cdot y) \in Y^- \times Y | g\in G; y\in \mu^{-1}(0)\},$$
which is embedded if $G$ acts freely on $Z \coloneqq \mu^{-1}(0)$ with smooth symplectic quotient $\bar{Y} \coloneqq Y\sslash_0 G = Z/G$. In this case, $L'$ can also be written as a fiber product $Z \times_{\bar{Y}} Z$.
\end{example}    

\begin{example}\label{example:orbit-lagrangian}
    Let $(Y = \mathcal{O}_\lambda, \omega)$ be a coadjoint orbit of $G$ with the Kirillov-Kostant-Souriau symplectic form $\omega$, 
    for which $G$ acts in a Hamiltonian fashion via coadjoint action with moment map $\mu = \iota: \mathcal{O}_\lambda \hookrightarrow \mathfrak{g}^*$. 
    This time, we consider the transpose of the moment Lagrangian:
    $$L^t_\mu \coloneqq \{(y, g\cdot y, g, \mu(y))) \subset ((\mathcal{O}_\lambda^- \times \mathcal{O}_\lambda)^- \times (G \times \mathfrak{g}^*), -(-\omega\oplus\omega)\oplus\Omega)$$ 
    as a Lagrangian correspondence from $\mathcal{O}_\lambda^- \times \mathcal{O}_\lambda$ to $G \times \mathfrak{g}^*$. Let $L = \Delta_{\mathcal{O}_\lambda}$ be the diagonal Lagrangian, 
    then the geometric transformation of $L$ by $L^t_\mu$ is given by the character Lagrangian:
    $$L_{\mathcal{O}_\lambda} \coloneqq L^t_\mu \circ L =  \{(g, \xi) \in G \times \mathfrak{g}^* | g\in G; \Ad^*_g(\xi) = \xi\in \mathcal{O}_\lambda\},$$
    which is embedded since $\mu$ is injective.
\end{example}

We now recall the following result of Fukaya on preservation of unobstructedness of Lagrangians under geometric transformations below:

\begin{theorem}\cite[Theorem 1.5]{Fukaya-corr} \label{unobstructedness}
Given the following datum:
\begin{enumerate}
    \item two closed or tame symplectic manifolds $(M, \omega_M)$, $(N, \omega_N)$;
    \item two real oriented bundles $U \rightarrow M$, $V \rightarrow N$;
    \item a closed, connected, oriented, $(p^*_M(U\oplus TM)\oplus p^*_NV)$-relatively spin and unobstructed Lagrangian correspondence $\mathcal{L} \subset (M^-\times N, \omega_M\oplus \omega_N)$;
    \item a bounding cochain $b_{\mathcal{L}}$ of $\mathcal{L}$.
\end{enumerate} 
Then for any closed, connected, oriented, $U$-relatively spin and unobstructed Lagrangian $L \subset (M, \omega_M)$ with a bounding cochain $b$ such that the geometric transformation
$L' \coloneqq \mathcal{L} \circ L$ is also closed, connected, then $L'$ is also oriented, $V$-relatively spin and unobstructed with a bounding cochain $b'$ canonically determined by $b$ and $b_{\mathcal{L}}$.
\end{theorem}

\begin{remark}
This can be generalised to weak bounding cochains, i.e. $b$ and $b_{\mathcal{L}}$ are weak bounding cochains, then $b'$ is also a weak bounding cochain satisfying an identity of their disc potentials. Interested readers are referred to \cite[Proposition 2.10]{lauleungli} for details. Meanwhile, Theorem \ref{unobstructedness} would suffice for applications in this paper.
\end{remark}

\begin{remark}
    Readers are cautioned that tautological unobstructedness needs not be preserved under geometric transformations: $b = 0 = b_{\mathcal{L}}$ does not imply $b' = 0$, unless the cyclic element in the correspondence tri-module is closed under the tri-module differential. See also \cite[Lemma 7.21]{gao} for relevant discussion in the exact case. 
\end{remark}

We are now ready to prove Corollary \ref{cor: exunoblag} on the unobstructedness of the geometric transformations 
in Examples \ref{example:shift-diagonal} 
and \ref{example:orbit-lagrangian} are unobstructed using Theorem \ref{unobstructedness}:

\begin{proof}[Proof of Corollary \ref{cor: exunoblag}]
For (1), the assumptions guarantee that $L_\mu$ is tautologically unobstructed by 
Corollary \ref{momentlagrangianunobstructed}; also, $Z$ is closed and connected 
(by Atiyah-Guillemin-Sternberg convexity theorem), so as $L' = Z \times_{\bar{Y}} Z$. 

Concerning the spin assumptions in Theorem \ref{unobstructedness}, since both $L$, $L_\mu$ and $T^*G$ are spin, so as the embedded composition $L' = L_\mu \circ L$. Hence, we may choose both $U$ and $V$ to be trivial line bundle, so that the spin assumptions are satisfied.

Therefore, since $L$ is exact (and hence tautologically unobstructed), Theorem \ref{unobstructedness} implies that $L'$ is unobstructed.\\

For (2), again we check the spin assumptions in Theorem \ref{unobstructedness}: since we assume $\mathcal{O}_\lambda$ is spin, so as $\Delta_{\mathcal{O}_\lambda}$ (which is in particular $pr^*_1(T\mathcal{O}_\lambda)$-relatively spin, where $pr_1: \mathcal{O}^-_\lambda \times \mathcal{O}_\lambda \rightarrow \mathcal{O}_\lambda$ is the projection to the first factor). Therefore, since $L^t_\mu$ and $\mathcal{O}^-_\lambda \times \mathcal{O}_\lambda$ are spin, so as $L_{\mathcal{O}_\lambda}$. Hence, we may choose $U = pr^*_1(T\mathcal{O}_\lambda)$ and $V$ to be trivial line bundle so that the spin assumptions are satisfied.

Therefore, since both $L^t_\mu$ and $\Delta_{\mathcal{O}_\lambda}$ are tautologically unobstructed (where the latter holds by \cite[Theorem 1.9]{FOOOinvol}), Theorem \ref{unobstructedness} again implies that $L_{\mathcal{O}_\lambda}$ is unobstructed.
\end{proof}

\appendix

\section{Symplectic structure of $G \times \mathfrak{g}^*$} \label{appendix:symplecticform}
Here we provide explicit formula for the symplectic form $\Omega$ on $G \times \mathfrak{g}^*$ under the identification with $(T^*G, -\omega_{can})$ via left translation, following \cite[Example 14.7]{CFM}. This was used in the proof of Theorem \ref{twisted-groupoid}. 

Given any compact connected Lie group $G$, there exists a diffeomorphism 
\begin{align*}
l: G \times \mathfrak{g}^* &\rightarrow T^*G\\
(g, \xi) &\mapsto \mathcal{L}^*_g(\xi),
\end{align*}
where $\mathcal{L}^*_g: \mathfrak{g}^* \rightarrow T^*_gG$ is dual to the inverse of the left translation $\mathcal{L}^{-1}_{g}: T_gG\rightarrow \mathfrak{g}$.

The symplectic form $\Omega$ on $G \times \mathfrak{g}^*$ is given by $\Omega = -l^*(\omega_{can})$. More explicitly, for any $(g, \xi) \in G \times \mathfrak{g}^*$ and $(Z, \eta), (W, \zeta) \in T_{(g, \xi)}(G \times \mathfrak{g}^*) \cong T_gG \times \mathfrak{g}^*$, we have
\begin{align}\label{symplecticform}
\Omega((Z, \eta), (W, \zeta)) = \langle \eta, \mathcal{L}^{-1}_{g}(W) \rangle - \langle \zeta, \mathcal{L}^{-1}_{g}(Z) \rangle - \langle \xi, [\mathcal{L}^{-1}_{g}(Z), \mathcal{L}^{-1}_{g}(W)] \rangle.
\end{align}

Using this symplectic form $\Omega$, the action groupoid $G \times \mathfrak{g}^*$ of the coadjoint $G$-action on $\mathfrak{g}^*$ is a symplectic groupoid $(G \times \mathfrak{g}^*, \Omega) \xRightarrow[Ad^*]{pr} \mathfrak{g}^*$. See e.g. \cite[Example 14.7]{CFM} for more details.

\section{Lie groupoids}\label{appendix:liegroupoid}
We follow the convention of Lie groupoids in \cite{CFM}, recalled below for reader's convenience:
\begin{defn} \cite[Definition 13.1]{CFM} \label{liegroupoid}
A Lie groupoid consists of a pair of smooth manifolds $(X, M)$ and a five-tuple of smooth maps $(s,t, m, e, i)$, where 
\begin{itemize}
\item (source and target maps) $s,t: X \rightarrow M$ are surjective submersions.
\item (multiplication map) $m: X^{(2)} \rightarrow X$ is a smooth map, where 
$$ X^{(2)} \coloneqq \{(p,q) \in X \times X: s(p) = t(q)\},$$
which is smooth since $s$ and $t$ are submersions. Denote $p\cdot q \coloneqq m(p,q)$.
\item (unit map) $e: M \rightarrow X$ is a smooth embedding.
\item (inversion map) $i: X \rightarrow X$ is a smooth involution, i.e. $i\circ i=\Id$. We write $p^{-1} \coloneqq i(p)$.
\end{itemize}
such that for each $(p,q),(q,r) \in X^{(2)}$ and $x\in M$, the following conditions are satisfied:
\begin{itemize}
\item $e$ is a section of $s$ and $t$: $s(e(x)) = x = t(e(x))$.
\item $e(t(p)) \cdot p  = p =  p \cdot e(s(p)) $
\item $s(p \cdot q) = s(q) ;  t(p \cdot q) =t(p)$
\item $(p \cdot q) \cdot r = p \cdot (q \cdot r)$.
\item $s(p^{-1}) = t(p) ; t(p^{-1}) = s(p)$.
\item $p \cdot p^{-1} = e(t(p)); p^{-1} \cdot p = e(s(p))$.
 
\end{itemize}
For simplicity, we will denote a Lie groupoid by $(X \xRightarrow[t]{s} M, m, e, i)$.\\
\end{defn}

\bibliographystyle{plain}
\bibliography{geometry}

@inproceedings {Teleman,
	AUTHOR = {Teleman, C.},
	TITLE = {Gauge theory and mirror symmetry},
	BOOKTITLE = {Proceedings of the {I}nternational {C}ongress of
	{M}athematicians---{S}eoul 2014. {V}ol. {II}},
	PAGES = {1309--1332},
	PUBLISHER = {Kyung Moon Sa, Seoul},
	YEAR = {2014},
}

@article {Oh,
    AUTHOR = {Oh, Y.-G.},
     TITLE = {Floer cohomology of {L}agrangian intersections and
              pseudo-holomorphic disks. {I}},
   JOURNAL = {Comm. Pure Appl. Math.},
  FJOURNAL = {Communications on Pure and Applied Mathematics},
    VOLUME = {46},
      YEAR = {1993},
    NUMBER = {7},
     PAGES = {949--993},
}

@BOOK{FOOO,
  title = {Lagrangian intersection {F}loer theory: anomaly and obstruction.
	{P}art {I} and {II}},
  publisher = {American Mathematical Society},
  year = {2009},
  author = {Fukaya, K. and Oh, Y.-G. and Ohta, H. and Ono, K.},
  volume = {46},
  series = {AMS/IP Studies in Advanced Mathematics},
  address = {Providence, RI},
  isbn = {978-0-8218-4836-4},
}

@ARTICLE{EL,
    AUTHOR = {Evans, J. and Lekili, Y. },
     TITLE = {Generating the {F}ukaya categories of {H}amiltonian
              {$G$}-manifolds},
   JOURNAL = {J. Amer. Math. Soc.},
  FJOURNAL = {Journal of the American Mathematical Society},
    VOLUME = {32},
      YEAR = {2019},
    NUMBER = {1},
     PAGES = {119--162},
      ISSN = {0894-0347,1088-6834},
       DOI = {10.1090/jams/909},
       URL = {https://doi.org/10.1090/jams/909},
}

@ARTICLE{Fukaya-corr,
	author = {Fukaya, K.},
	title = {{U}nobstructed immersed {L}agrangian correspondence and filtered {$A_\infty$} functor},
	journal = {preprint},
	note = {\href{https://arxiv.org/abs/1706.02131}{arxiv.org/abs/1706.02131}}
}

@book {foookuranishi,
    AUTHOR = {Fukaya, K. and Oh, Y.-G. and Ohta, H. and Ono,
              K.},
     TITLE = {Kuranishi structures and virtual fundamental chains},
    SERIES = {Springer Monographs in Mathematics},
 PUBLISHER = {Springer, Singapore},
      YEAR = {[2020] \copyright 2020},
     PAGES = {xv+638},
      ISBN = {978-981-15-5562-6; 978-981-15-5561-9},
   MRCLASS = {53D45 (32M99 53D37 53D40)},
  MRNUMBER = {4179586},
       DOI = {10.1007/978-981-15-5562-6},
       URL = {https://doi-org.easyaccess2.lib.cuhk.edu.hk/10.1007/978-981-15-5562-6},
}

@article{gao,
	Author = {Gao, Y.},
	Journal = {preprint},
	Note = {\href{https://arxiv.org/abs/1712.00225}{arXiv:1712.00225}},
	Title = {Functors of wrapped {F}ukaya categories from {L}agrangian correspondences},
}

@article{xiao2,
    AUTHOR ={Xiao, Y.},
	Journal = {preprint},
	Note = {\href{https://arxiv.org/abs/2405.11169}{arXiv:2405.11169}},
	Title = {{M}oment {L}agrangian correspondences are unobstructed after bulk deformation},
}

@book {CFM,
    AUTHOR = {Crainic, M. and Fernandes, R. L. and M\u arcu\c t,
              I.},
     TITLE = {Lectures on {P}oisson geometry},
    SERIES = {Graduate Studies in Mathematics},
    VOLUME = {217},
 PUBLISHER = {American Mathematical Society, Providence, RI},
      YEAR = {[2021] \copyright 2021},
     PAGES = {xix+479},
      ISBN = {978-1-4704-6430-1},
       DOI = {10.1090/gsm/217},
       URL = {https://doi.org/10.1090/gsm/217},
}

@article {FOOOinvol,
    AUTHOR = {Fukaya, K. and Oh, Y.-G. and Ohta, H. and Ono,
              K.},
     TITLE = {Antisymplectic involution and {F}loer cohomology},
   JOURNAL = {Geom. Topol.},
  FJOURNAL = {Geometry \& Topology},
    VOLUME = {21},
      YEAR = {2017},
    NUMBER = {1},
     PAGES = {1--106},
      ISSN = {1465-3060,1364-0380},
       DOI = {10.2140/gt.2017.21.1},
       URL = {https://doi-org.libproxy.ibs.re.kr/10.2140/gt.2017.21.1},
}

@ARTICLE{lauleungli,
  author = {Lau, S.-C. and Leung, N.C. and Li, Y. L.},
  title = {Equivariant {L}agrangian correspondence and a conjecture of {T}eleman},
journal = {preprint},
    YEAR = {2023},
	note = {\href{https://arxiv.org/abs/2312.13926}{https://arxiv.org/abs/2312.13926}}
}

@article {Pascaleffmonoidal,
    AUTHOR = {Pascaleff, J.},
     TITLE = {Poisson geometry, monoidal {F}ukaya categories, and
              commutative {F}loer cohomology rings},
   JOURNAL = {Enseign. Math.},
  FJOURNAL = {L'Enseignement Math\'ematique},
    VOLUME = {70},
      YEAR = {2024},
    NUMBER = {3-4},
     PAGES = {313--381},
      ISSN = {0013-8584,2309-4672},
   MRCLASS = {53D37 (18M05 53D17 53D40)},
  MRNUMBER = {4795722},
MRREVIEWER = {Cheuk\ Yu\ Mak},
       DOI = {10.4171/lem/1071},
       URL = {https://doi-org.libproxy.ibs.re.kr/10.4171/lem/1071},
}

@article {Weinsteinsympgeom,
    AUTHOR = {Weinstein, A.},
     TITLE = {Symplectic geometry},
   JOURNAL = {Bull. Amer. Math. Soc. (N.S.)},
  FJOURNAL = {American Mathematical Society. Bulletin. New Series},
    VOLUME = {5},
      YEAR = {1981},
    NUMBER = {1},
     PAGES = {1--13},
      ISSN = {0273-0979,1088-9485},
   MRCLASS = {58F05 (70G15)},
  MRNUMBER = {614310},
MRREVIEWER = {Gunther\ A.\ Uhlmann},
       DOI = {10.1090/S0273-0979-1981-14911-9},
       URL = {https://doi-org.libproxy.ibs.re.kr/10.1090/S0273-0979-1981-14911-9},
}

@incollection {sympgroupoid,
    AUTHOR = {Coste, A. and Dazord, P. and Weinstein, A.},
     TITLE = {Groupo\"ides symplectiques},
 BOOKTITLE = {Publications du {D}\'epartement de {M}ath\'ematiques.
              {N}ouvelle {S}\'erie. {A}, {V}ol.\ 2},
    SERIES = {Publ. D\'ep. Math. Nouvelle S\'er. A},
    VOLUME = {87-2},
     PAGES = {i--ii, 1--62},
 PUBLISHER = {Univ. Claude-Bernard, Lyon},
      YEAR = {1987},
   MRCLASS = {58F05 (17B99 20L99 58F06)},
  MRNUMBER = {996653},
MRREVIEWER = {Charles-Michel\ Marle},
}

@incollection {sympcategory,
    AUTHOR = {Weinstein, A.},
     TITLE = {The symplectic ``category''},
 BOOKTITLE = {Differential geometric methods in mathematical physics
              ({C}lausthal, 1980)},
    SERIES = {Lecture Notes in Math.},
    VOLUME = {905},
     PAGES = {45--51},
 PUBLISHER = {Springer, Berlin-New York},
      YEAR = {1982},
      ISBN = {3-540-11197-2},
   MRCLASS = {58F05 (58F06)},
  MRNUMBER = {657441},
MRREVIEWER = {Gunther\ A.\ Uhlmann},
}

@book {GSsemiclassical,
    AUTHOR = {Guillemin, V. and Sternberg, S.},
     TITLE = {Semi-classical analysis},
 PUBLISHER = {International Press, Boston, MA},
      YEAR = {2013},
     PAGES = {xxiv+446},
      ISBN = {978-1-57146-276-3},
   MRCLASS = {58-02 (53-02 53D05 53D12 58J40 81Q10)},
  MRNUMBER = {3157301},
MRREVIEWER = {Alberto\ Parmeggiani},
}

@article{Chowmomentcorr,
	Author = {C. H. Chow},
	Journal = {Preprint},
	Note = {\href{https://arxiv.org/pdf/2107.08576}{arXiv:2107.08576}},
  YEAR = {2021},
	Title = {Quantum characteristic classes, moment correspondences and the {H}amiltonian groups of coadjoint orbits},
}

@article{BCL,
	Author = {Bae, H. and Chow, C. H. and Leung, N. C.},
	Journal = {Preprint},
	Note = {\href{https://arxiv.org/pdf/2103.00382}{arXiv:2103.00382}},
  YEAR = {2021},
	Title = {Applications of the theory of Floer to symmetric spaces},
}

@article{AFOOO,
	Author = {Abouzaid, M. and Fukaya, K. and Oh, Y.-G. and Ohta, H. and Ono,
	K.},
	Journal = {Preprint},
	Note = {\href{https://arxiv.org/pdf/2606.12257}{arXiv:2606.12257}},
  YEAR = {2026},
	Title = {Quantum cohomology and split generation in {L}agrangian {F}loer theory},
}

@article {weinsteinsympgroupoid,
    AUTHOR = {Weinstein, A.},
     TITLE = {Symplectic groupoids and {P}oisson manifolds},
   JOURNAL = {Bull. Amer. Math. Soc. (N.S.)},
  FJOURNAL = {American Mathematical Society. Bulletin. New Series},
    VOLUME = {16},
      YEAR = {1987},
    NUMBER = {1},
     PAGES = {101--104},
      ISSN = {0273-0979,1088-9485},
   MRCLASS = {58F05 (20L15 58H05)},
  MRNUMBER = {866024},
       DOI = {10.1090/S0273-0979-1987-15473-5},
       URL = {https://doi-org.libproxy.ibs.re.kr/10.1090/S0273-0979-1987-15473-5},
}

@InProceedings{GSmomentred,
author="Guillemin, V.
and Sternberg, S.",
title="Moments and reductions",
booktitle="Differential Geometric Methods in Mathematical Physics",
year="1982",
publisher="Springer Berlin Heidelberg",
address="Berlin, Heidelberg",
isbn="978-3-540-39002-2"
}

@article {CWsympcat,
    AUTHOR = {Crooks, P. and Weitsman, J.},
     TITLE = {Towards a quantization of the double via the enhanced
              symplectic ``category''},
   JOURNAL = {Res. Math. Sci.},
  FJOURNAL = {Research in the Mathematical Sciences},
    VOLUME = {8},
      YEAR = {2021},
    NUMBER = {2},
     PAGES = {Paper No. 32, 51},
      ISSN = {2522-0144,2197-9847},
       DOI = {10.1007/s40687-021-00269-0},
       URL = {https://doi-org.libproxy.ibs.re.kr/10.1007/s40687-021-00269-0},
}

@article {AbouzaidDiogo,
    AUTHOR = {Abouzaid, M. and Diogo, L.},
     TITLE = {Monotone {L}agrangians in cotangent bundles of spheres},
   JOURNAL = {Adv. Math.},
  FJOURNAL = {Advances in Mathematics},
    VOLUME = {427},
      YEAR = {2023},
     PAGES = {Paper No. 109114, 45},
      ISSN = {0001-8708,1090-2082},
       DOI = {10.1016/j.aim.2023.109114},
       URL = {https://doi-org.libproxy.ibs.re.kr/10.1016/j.aim.2023.109114},
}

@article {AlbersFrauenfelder,
    AUTHOR = {Albers, P. and Frauenfelder, U.},
     TITLE = {A nondisplaceable {L}agrangian torus in {$T^*S^2$}},
   JOURNAL = {Comm. Pure Appl. Math.},
  FJOURNAL = {Communications on Pure and Applied Mathematics},
    VOLUME = {61},
      YEAR = {2008},
    NUMBER = {8},
     PAGES = {1046--1051},
      ISSN = {0010-3640,1097-0312},
   MRCLASS = {53D40 (53D12)},
       DOI = {10.1002/cpa.20216},
       URL = {https://doi-org.libproxy.ibs.re.kr/10.1002/cpa.20216},
}

@article {OU,
    AUTHOR = {Oakley, J. and Usher, M.},
     TITLE = {On certain {L}agrangian submanifolds of {$S^2\times S^2$} and
              {$\Bbb{C}{\rm P}^n$}},
   JOURNAL = {Algebr. Geom. Topol.},
  FJOURNAL = {Algebraic \& Geometric Topology},
    VOLUME = {16},
      YEAR = {2016},
    NUMBER = {1},
     PAGES = {149--209},
      ISSN = {1472-2747,1472-2739},
   MRCLASS = {53D12 (53D40)},
       DOI = {10.2140/agt.2016.16.149},
       URL = {https://doi-org.libproxy.ibs.re.kr/10.2140/agt.2016.16.149},
}

@article{BP,
	Author = {Brendel, J. and Piccirilli, A.},
	Journal = {preprint},
	Note = {\href{https://arxiv.org/pdf/2607.15136}{arXiv:2607.15136}},
	Title = {On Lagrangians of Oakley–Usher type},
}
	
\end{document}